\documentclass[11pt]{article}

\usepackage[a4paper,margin=1in]{geometry}
\usepackage[T1]{fontenc}
\usepackage[utf8]{inputenc}
\usepackage{lmodern}
\usepackage{amsmath,amsthm,amssymb,mathtools}
\numberwithin{equation}{section}
\usepackage{enumitem}
\usepackage{lineno}
\usepackage[colorlinks=true,linkcolor=blue,citecolor=blue,urlcolor=blue]{hyperref}

\newtheorem{theorem}{Theorem}[section]
\newtheorem{lemma}[theorem]{Lemma}
\newtheorem{corollary}[theorem]{Corollary}
\newtheorem{proposition}[theorem]{Proposition}
\newtheorem{problem}[theorem]{Problem}

\theoremstyle{definition}

\newtheorem{remark}[theorem]{Remark}

\DeclareMathOperator{\ex}{ex}
\DeclareMathOperator{\pal}{pal}
\DeclareMathOperator{\revop}{rev}
\newcommand{\piu}{\pi_{\mathrm u}}
\newcommand{\R}{\mathbb R}
\newcommand{\Cvec}[1]{\overrightarrow{C}_{#1}}
\newcommand{\Fchain}{\overrightarrow{\mathcal F}}
\newcommand{\A}{\mathcal A}
\newcommand{\C}{\mathcal C}

\title{Palette extremality via degree-square Tur\'an problems}

\author{
Jiangdong Ai\thanks{School of Mathematical Sciences and LPMC, Nankai University, Tianjin 300071, China. Email: \texttt{jd@nankai.edu.cn}. Supported by the National Natural Science Foundation of China (No.12522117).}
\and
Bin Chen\thanks{School of Mathematics and Statistics, Fuzhou University, Fujian, China. Email: \texttt{cbfzu03@163.com}. Supported by the National Natural Science Foundation of China (Nos.12501473 and 12471336).}
\and
Ming Chen\thanks{School of Mathematics and Statistics, Jiangsu Normal University, Xuzhou, China. Email: \texttt{chenming314@jsnu.edu.cn}. Supported by the National Key Research and Development Program of China (No.2024YFA1013900), the Basic Research Program of Jiangsu (No.BK20251044), the National Natural Science Foundation of China (No.12501483), and the Natural Science Foundation of the Jiangsu Higher Education Institutions of China (No.25KJB110003).}
\and
Zilong Yan\thanks{School of Mathematics, Hunan University,
 Changsha, China. Email: \texttt{zilongyan@hnu.edu.cn}.}
\and
Tianxiao Zhao\thanks{School of Mathematics, Harbin Institute of Technology, Harbin, China. Email: \texttt{zhaotianxiao@hit.edu.cn}. Supported by the National Natural Science Foundation of China (No.12501475).}
}

\begin{document}
\maketitle

\begin{abstract}
We study finite palette extremal problems motivated by uniform Tur\'an
densities of $3$-uniform hypergraphs. Given a self-converse tournament $T$ with at least two vertices, we determine
the largest possible number of admissible triples in an $m$-color palette that avoids the left and right palettes associated with $T$.  The answer is the one-sided degree-square Tur\'an number
\[
\operatorname{pal}_T(m)
=
\operatorname{ex}_2^+(m,T)
=
\max\left\{
\sum_{v\in V(D)} d_D^+(v)^2:
|V(D)|=m,\ D\text{ is }T\text{-free}
\right\}.
\]
Thus this palette problem is reduced to an extremal problem for digraph
out-degrees.

We then prove a prefix-majorization lemma for convex out-degree moments and
apply it to directed cycles.  In particular, $\operatorname{ex}_2^+(m,\overrightarrow C_3)=\frac{m(m^2-1)}3$,
which gives the sharp $m$-color palette endpoint $\frac13-\frac{1}{3m^2}$ for the directed triangle.  Combining this endpoint with the palette characterization of uniform Tur\'an density and the palette separation theorem,
we show that for every $m\ge2$ there is a finite $3$-graph $H_m$ such that
\[
\frac13-\frac{1}{3m^2}\le \pi_u(H_m)\le \frac13.
\]
Hence there is a sequence of individual finite $3$-graphs whose uniform Tur\'an densities converge to $1/3$.

We also describe the extremal palettes, prove a qualitative edit-distance
stability theorem, and compute the Lagrangian of the endpoint palette
$\mathcal P_m^\star$.  As a consequence, for every $m\ge2$ there is a
finite family $\mathcal F_m$ of $3$-graphs with $\pi_u(\mathcal F_m)=\frac13-\frac{1}{3m^2}$,
so $1/3$ is an accumulation point of uniform Tur\'an densities of finite forbidden families.
\end{abstract}

\section{Introduction}

Tur\'an-type problems ask for the largest possible density of a large
combinatorial structure avoiding a prescribed configuration.  For ordinary
graphs, the asymptotic theory is governed by the Erd\H{o}s--Stone--Simonovits
theorem~\cite{ErdosSimonovits1966,ErdosStone1946}.  Hypergraph Tur\'an theory
is much less rigid: even for $3$-uniform hypergraphs, several basic density
problems remain open.

The present paper concerns the uniform-density version introduced by
Erd\H{o}s and S\'os~\cite{ErdosSos1982}.  Let $H$ be a finite $3$-uniform
hypergraph.  For $d\in[0,1]$ and $\eta>0$, an $n$-vertex $3$-graph $G$ is
called $(d,\eta)$-\textit{uniformly dense} if every vertex subset
$U\subseteq V(G)$ with $|U|\ge \eta n$ satisfies
 $e_G(U)\ge d\binom{|U|}{3}$.
 
The \textit{uniform Tur\'an density} of $H$, denoted by $\piu(H)$, is the
supremum of all $d\in[0,1]$ such that for every $\eta>0$ and every $n_0$
there exists an $H$-free $(d,\eta)$-uniformly dense $3$-graph on some
$n\ge n_0$ vertices.  
In other words, $\piu(H)$ measures the largest edge density that
can be forced inside every linear-size vertex subset while still avoiding
$H$.  The first nontrivial exact value in this setting was the density of
the broken tetrahedron $K_4^{(3)-}$, determined independently by Glebov,
Kr\'al' and Volec~\cite{GlebovKralVolec2016} and by Reiher, R\"odl and
Schacht~\cite{ReiherRodlSchacht2018}.  Further exact results include the
uniform Tur\'an densities of tight cycles, obtained by Buci\'c, Cooper,
Kr\'al', Mohr and Munh\'a Correia~\cite{BucicCooperKralMohrMunha2023}.

A central tool in this area is the palette method.  A \textit{palette} is a
pair $\mathcal P=(\C,\A)$, where $\C$ is a finite set of colors and
$\A\subseteq \C^3$ is a set of admissible ordered triples.  Given a linear
order on a vertex set and a coloring $\chi:\binom{V}{2}\to \C$, a triple
$xyz$ with $x<y<z$ is declared to be an edge whenever
$\bigl(\chi(xy),\chi(xz),\chi(yz)\bigr)\in\A$.

A $3$-graph $H$ is $\mathcal P$-\textit{colorable} if some ordering of
$V(H)$ and some coloring of all pairs by colors from $\C$ make every edge of
$H$ admissible in this sense.  The \textit{density} of $\mathcal P$ is $d(\mathcal P)=\frac{|\A|}{|\C|^3}$.
If $H$ is not $\mathcal P$-colorable, then a standard probabilistic construction based on $\mathcal P$ gives the lower bound $\piu(H)\ge d(\mathcal P)$.

Recent work has shown that these palette lower bounds are, in a precise
sense, exhaustive.  Lamaison~\cite{Lamaison2024} proved that $\piu(H)$ is
determined by the densities of palettes that do not color $H$.  Kr\'al',
Ku\v{c}er\'ak, Lamaison and Tardos~\cite{KKLT2025} then proved a palette
separation theorem, giving necessary and sufficient conditions for the
existence of a finite $3$-graph satisfying prescribed palette colorability
and non-colorability constraints.  Related work of King, Piga, Sales and
Sch\"ulke~\cite{KingPigaSalesSchulke2025} studied possible uniform Tur\'an
densities through palette Lagrangians. Very recently, Liu and
Pikhurko~\cite{LiuPikhurko2026} proved that the uniform Tur\'an densities of
finite families of $3$-graphs are dense in a nondegenerate interval.

These developments reduce many questions about uniform Tur\'an density to
finite extremal problems for palettes.  The guiding problem in this paper is:
given a fixed number of colors and a collection of forbidden palettes,
how many admissible triples can a palette have while avoiding all palettes in the
collection?
We solve this problem for a natural
family of left and right tournament palettes, and identify the answer with a one-sided degree-square
Tur\'an problem for digraphs. Our results complement the interval theorem of Liu and Pikhurko by giving an explicit sequence converging to the specific
point $1/3$, together with exact extremal structures, stability, and a palette-Lagrangian calculation.

\subsection{Relation to previous palette--digraph reductions}

Several recent papers have connected palette extremal problems with digraph
extremal theory.  Lamaison and Wu~\cite{LamaisonWu2024} used an auxiliary
digraph associated with a palette in their work on the uniform Tur\'an density
of large stars.  Lin and Zhou~\cite{LinZhou2025} developed this approach for
stars in uniformly dense hypergraphs, using auxiliary digraphs together with
Caro--Wei type degree estimates.  Lin, Wang, Zhou and Zhou~\cite{LinWangZhouZhou2026}
subsequently introduced a systematic connection between digraph Tur\'an
problems and uniform Tur\'an densities of $3$-graphs, including left and right
palettes associated with digraphs.  Very recently, Lin, Sun, Wang and
Zhou~\cite{LinSunWangZhou2026} extended the palette framework to certain $k$-uniform settings.

Our contribution is a degree-square refinement of this line of work. For the left and right palettes generated by a self-converse tournament, we prove an
exact identity between finite palette extremality and a degree-square Tur\'an problem for digraphs. We then compute the relevant degree-square Tur\'an
numbers for directed cycles through a prefix-majorization principle for convex out-degree moments.

The new point is that the relevant palette extremal problem is not governed
only by the ordinary number of arcs in an auxiliary digraph, but by a
one-sided convex degree moment.  This allows us to obtain exact finite
endpoints, uniqueness of extremal palettes, edit-distance stability, and
exact palette Lagrangians for the cyclic-triangle endpoint.

\subsection{Main results}
Throughout the paper, digraphs are finite. Unless loops are explicitly
allowed, they are loopless, but digons (pairs of opposite arcs on the same two vertices) are permitted. An $F$-free digraph contains no copy of $F$ (not necessarily induced), and two opposite arcs are counted separately. We write $D^{\mathrm{rev}}$ for the digraph obtained by reversing every arc of $D$, and we call $D$ \emph{self-converse} if $D\cong D^{\mathrm{rev}}$.

For a loopless digraph $F$ and a nondecreasing convex function
$\varphi:[0,n-1]\to\R$, define
\[
\operatorname{ex}_{\varphi}^{+}(n,F)
=
\max\left\{
\sum_{v\in V(D)}\varphi(d_D^+(v)):
|V(D)|=n,\; D\text{ is }F\text{-free}
\right\}.
\]
In particular,
$\operatorname{ex}_2^+(n,F)=\operatorname{ex}_{x^2}^{+}(n,F)$.

Let $D$ be a loopless digraph, and let
\[
\C_D=V(D)\sqcup\{c_{uv}:uv\in A(D)\},
\]
where $c_{uv}$ are new colors, one for each arc $uv$. The left
and right palettes generated by $D$ are
\[
\mathcal{P}_D^L=(\C_D,\{(u,v,c_{uv}):uv\in A(D)\})
\]
and
\[
\mathcal{P}_D^R=(\C_D,\{(c_{uv},u,v):uv\in A(D)\}).
\]
A palette \textit{homomorphism} $\mathcal P\to \mathcal P'$ is a map between
color sets preserving admissible triples, and we write $\mathcal P\preceq
\mathcal P'$ when such a homomorphism exists.
For a palette $\mathcal P$,
we write $\C(\mathcal P)$ and $\A(\mathcal P)$ for the set of colors and the set of
admissible triples respectively.
For a tournament $T$, let
\[
\pal_T(m)=
\max\{ |\A(\mathcal P)|:
|\C(\mathcal P)|=m,
\mathcal{P}_T^L\npreceq \mathcal P,\ \text{and}\ 
\mathcal{P}_T^R\npreceq \mathcal P\}.
\]
Our first main theorem identifies this finite palette extremal problem with a
one-sided degree-square Tur\'an problem.

\begin{theorem}\label{thm:intro-palette}
Let $T$ be a self-converse tournament with at least two vertices. Then, for every $m\ge 1$,
$\pal_T(m)=\operatorname{ex}_2^+(m,T)$, where
\[
\operatorname{ex}_2^+(m,T)
=
\max
\left\{
\sum_{v\in V(D)} d_D^+(v)^2:
|V(D)|=m,\; D\text{ is }T\text{-free}
\right\}.
\]
\end{theorem}

The proof associates a palette with two auxiliary digraphs
$G_L(\mathcal P)$ and $G_R(\mathcal P)$ (to be defined later).
The non-existence of homomorphisms
from $\mathcal{P}_T^L$ and $\mathcal{P}_T^R$ forces these two auxiliary digraphs
to be $T$-free, and Cauchy--Schwarz reduces the palette bound to the
corresponding degree-square bound.  The reverse inequality is obtained by
constructing a palette directly from an extremal $T$-free digraph.

We next give the degree-square input needed for the cyclic case.  In
Section~\ref{sec:majorization},  we prove a prefix-majorization principle for convex out-degree moments in $F$-free digraphs.  Applying it to the
Brown--Harary \cite{BrownHarary1970} and Zhou--Li \cite{ZhouLi2023} extremal digraphs for directed cycles gives the
following exact formula.

\begin{theorem}\label{thm:intro-cycles}
Let $\ell\ge 3$, $n=(\ell-1)q+r$ where $0\le r<\ell-1$, and let $\varphi:[0,n-1]\to\mathbb R$ be nondecreasing and convex. Then
\[
\operatorname{ex}_\varphi^+(n,\Cvec{\ell})
=
\sum_{i=1}^{q}
(\ell-1)\,
\varphi\!\left((\ell-1)(q-i+1)-1+r\right)
+
\begin{cases}
r\varphi(r-1), & r>0,\\
0, & r=0.
\end{cases}
\]
\end{theorem}

For $\varphi(x)=x^2$ and $\ell=3$, Theorem~\ref{thm:intro-cycles} gives
$\operatorname{ex}_2^+(m,\Cvec{3})=\frac{m(m^2-1)}3$.
Combining this with Theorem~\ref{thm:intro-palette} yields the sharp finite
endpoint for the two cyclic-triangle palettes:
$\pal_{\Cvec{3}}(m)=\frac{m(m^2-1)}3$, and $
\max d(\mathcal P)=\frac13-\frac1{3m^2}$, where the maximum is over all $m$-color palettes avoiding both
$\mathcal{P}_{\Cvec{3}}^L$ and $\mathcal{P}_{\Cvec{3}}^R$.  For brevity, in the
cyclic-triangle endpoint discussion we write
$\mathcal{P}_L=\mathcal{P}_{\Cvec{3}}^L$, and $\mathcal{P}_R=\mathcal{P}_{\Cvec{3}}^R$.

We also compute a second exact tournament input.  
Let $TT_\ell$ be the \textit{transitive tournament}
on $\ell$ vertices.
We have
$\operatorname{ex}_2^+(m,TT_3)=m\left\lfloor\frac{m^2}{4}\right\rfloor$, with equality attained by the balanced complete bipartite digraph.

The cyclic-triangle endpoint gives the following consequence for uniform
Tur\'an densities of individual finite $3$-graphs.  Its proof uses Lamaison's
characterization and the palette separation theorem.

\begin{theorem}\label{thm:intro-limit}
For every $m\ge 2$, there exists a finite $3$-uniform hypergraph $H_m$ such that
\[
\frac13-\frac1{3m^2}
\le
\piu(H_m)
\le
\frac13.
\]
Consequently, the densities $\piu(H_m)$ converge to $1/3$ as $m\to\infty$.
\end{theorem}

Thus Theorem~\ref{thm:intro-limit} gives a sequence of single forbidden $3$-graphs whose uniform Tur\'an densities converge to $1/3$. This statement is deliberately separate from the accumulation-point statement below: the Lagrangian calculation gives exact values converging to $1/3$ for finite forbidden families.

\subsection{Stability and Lagrangian consequences}

The equality analysis at the cyclic-triangle endpoint also gives structural
information.  We prove that the only extremal palettes are the endpoint
palettes coming from ordered digon chains, and we obtain a qualitative
edit-distance stability theorem.  More precisely, let
\[
E_m=\frac{m(m^2-1)}3=\operatorname{ex}_2^+(m,\Cvec{3}).
\]
For an ordered digon chain $J$ on a color set $\C$, the endpoint palette
$\mathcal{P}_m^\star(J)$ is defined by
\[
(x,y,z)\in\A(\mathcal{P}_m^\star(J))
\quad\Longleftrightarrow\quad
 y\to x\text{ and }y\to z\text{ in }J.
\]
When $J=\Fchain_{m,2}$ carries its standard ordering, we write simply
$\mathcal{P}_m^\star$.  The stability result, proved as
Theorem~\ref{thm:edit-stability}, says that if an $m$-color palette
$\mathcal P=(\C,\A)$ avoids both $\mathcal{P}_L$ and $\mathcal{P}_R$ and
satisfies
\[
|\A|\ge E_m-\varepsilon m^3,
\]
then, after choosing $\varepsilon$ sufficiently small in terms of a prescribed
$\eta>0$, the palette $\mathcal P$ differs in at most $\eta m^3$ admissible
triples from some endpoint palette $\mathcal{P}_m^\star(J)$.

Finally, we compute the palette Lagrangian of the endpoint palette.  
For a
palette $\mathcal P=(\C,\A)$, let
\[
\lambda_{\mathcal P}(\mathbf{x})
=
\sum_{(i,j,k)\in\A}x_i x_j x_k,
\qquad \mathbf{x}=(x_i)_{i\in\C},
\]
and let $\Lambda(\mathcal P)$ be the maximum of this polynomial subject to
$\sum_{i\in\C}x_i=1$ and $x_i\in[0,1]$ for all $i\in\C$.  We prove that, for
every $m\ge2$,
\[
\Lambda(\mathcal{P}_m^\star)=\frac13-\frac1{3m^2},
\]
and the unique maximizer is the uniform probability vector on the
$m$ colors. Together with a theorem of King, Piga, Sales and Sch\"ulke
\cite[Theorem~1.1]{KingPigaSalesSchulke2025}, this implies that for every $m\ge2$ there is a finite family $\mathcal F_m$ of
$3$-graphs such that
\[
\piu(\mathcal F_m)=\frac13-\frac1{3m^2}.
\]
Consequently, $1/3$ is an accumulation point of uniform Tur\'an densities of
finite forbidden families.

\subsection{Organization}

In Section~\ref{sec:majorization}, we prove the majorization principle for
convex out-degree moments in $F$-free digraphs.  In Section~\ref{sec:cycles},
we apply this principle to directed cycles.  In Section~\ref{sec:palettes},
we recall palettes, palette homomorphisms, left and right palettes, and the
auxiliary digraphs $G_L(\mathcal P)$ and $G_R(\mathcal P)$.  In
Section~\ref{sec:self-converse}, we prove the self-converse tournament
identity $\pal_T(m)=\operatorname{ex}_2^+(m,T)$; in
Subsection~\ref{subsec:transitive-triangle}, we compute the transitive-triangle
case.  In Section~\ref{sec:endpoint}, we specialize the result to the cyclic
triangle, characterize equality at the sharp palette endpoint, and prove
Theorem~\ref{thm:intro-limit}.  In Section~\ref{sec:edit-dist}, we prove the
edit-distance stability theorem; the proof uses a compactness argument via
digraph limits.  In Section~\ref{sec:lagrangian}, we compute the Lagrangian of
$\mathcal{P}_m^\star$ and deduce the accumulation-point consequence for finite
forbidden families.  We end with open problems in
Section~\ref{sec:questions}.

\section{Degree-majorization for digraphs}\label{sec:majorization}

For a loopless digraph $F$, let
\[
\ex(n,F)=\max\{|A(D)|: |V(D)|=n,\; D\text{ is }F\text{-free}\}
\]
denote the \textit{ordinary Tur\'an number}. For a nondecreasing convex function $\varphi:[0,n-1]\to \R$, define
\[
\operatorname{ex}_{\varphi}^{+}(n,F)
=
\max\left\{
\sum_{v\in V(D)}\varphi(d_D^+(v)):
|V(D)|=n,\; D\text{ is }F\text{-free}
\right\}.
\]
In particular,
$\operatorname{ex}_2^+(n,F)=\operatorname{ex}_{x^2}^{+}(n,F)$.

For disjoint vertex sets $X,Y\subseteq V(D)$, let
$e_D^+(X,Y)=|\{xy\in A(D):x\in X,\ y\in Y\}|$.
\begin{lemma}\label{lem:weak-majorization}
Let $a_1\ge a_2\ge \cdots\ge a_n\ge0$ and $b_1\ge b_2\ge \cdots\ge b_n\ge0$ be two nonincreasing sequences of integers bounded by $M$. Suppose that $\sum_{i=1}^j a_i\le \sum_{i=1}^j b_i$ for every $1\le j\le n$.
If $\varphi:[0,M]\to\R$ is nondecreasing and convex, then
\[
\sum_{i=1}^n \varphi(a_i)
\le
\sum_{i=1}^n \varphi(b_i).
\]
\end{lemma}

\begin{proof}
If $M=0$, then all entries are equal to $0$, and the statement is
immediate. Hence assume $M\ge1$.
For a nonincreasing sequence $c_1\ge c_2\ge\cdots\ge c_n$ and a real number $t$,
\[
\sum_{i=1}^n (c_i-t)_+
=
\max_{0\le j\le n}
\left(\sum_{i=1}^j c_i-jt\right).
\]
The prefix inequalities therefore imply
\[
\sum_i (a_i-t)_+\le \sum_i (b_i-t)_+
\qquad\text{for every }t.
\]
Since the entries are integers, every convex function on $\{0,1,\ldots,M\}$ can be written as
\[
\varphi(x)=\varphi(0)+\alpha x+
\sum_{t=1}^{M-1}\beta_t(x-t)_+,
\]
where $\alpha=\varphi(1)-\varphi(0)\ge0$ and $\beta_t=\varphi(t+1)-2\varphi(t)+\varphi(t-1)\ge0$. Taking $j=n$ in the prefix inequalities gives $\sum_i a_i\le \sum_i b_i$. Combining this with the inequalities for $(x-t)_+$ proves the lemma.
\end{proof}

\begin{theorem}\label{thm:majorization}
Let $F$ be a loopless digraph and let $n\ge1$. Suppose that there is an $F$-free digraph $G_n$ on $n$ vertices with nonincreasing out-degree sequence $b_1^{(n)}\ge b_2^{(n)}\ge\cdots\ge b_n^{(n)}$
such that
\begin{equation}\label{eq:prefix-extremal}
\sum_{i=1}^{j} b_i^{(n)}=\ex(j,F)+j(n-j)
\qquad\text{for every }1\le j\le n.
\end{equation}
Then, for every nondecreasing convex function $\varphi:[0,n-1]\to\R$,
$\operatorname{ex}_{\varphi}^{+}(n,F)=\sum_{i=1}^{n}\varphi(b_i^{(n)})$.
\end{theorem}

\begin{proof}
Let $D$ be an $F$-free digraph on $n$ vertices, and write its out-degrees in nonincreasing order as $d_1^+\ge d_2^+\ge\cdots\ge d_n^+$.
For $1\le j\le n$, let $U$ be a set of $j$ vertices of largest out-degree. Then
\[
\sum_{i=1}^{j} d_i^+
=
|A(D[U])|+e_D^+(U,V(D)\setminus U)
\le
\ex(j,F)+j(n-j)
=
\sum_{i=1}^{j} b_i^{(n)}
\] by \eqref{eq:prefix-extremal}.
Thus the out-degree sequence of $D$ is weakly majorized by that of $G_n$. Lemma~\ref{lem:weak-majorization} gives
\[
\sum_{v\in V(D)}\varphi(d_D^+(v))
\le
\sum_{i=1}^{n}\varphi(b_i^{(n)}).
\]
The digraph $G_n$ attains equality, proving the theorem.
\end{proof}

\begin{remark}\label{rem:in-degree}
If $F$ is isomorphic to its reverse, then the same bound holds for in-degree moments. Indeed, applying Theorem~\ref{thm:majorization} to the reverse of an $F$-free digraph bounds $\sum_v\varphi(d_D^-(v))$ by the same expression.
\end{remark}

\section{Directed cycles}\label{sec:cycles}

Let $\Cvec{\ell}$ be a directed cycle of length $\ell$. For $\ell\ge3$ and $n\ge1$, write $n=(\ell-1)q+r$ where $0\le r<\ell-1$.
Let $\Fchain_{n,\ell-1}$ be the ordered complete-block chain obtained by taking $q$ complete digraph blocks of size $\ell-1$, one final block of size $r$ if $r>0$, and orienting every inter-block arc from earlier blocks to later blocks. A complete digraph block means that both arcs are present between every pair of distinct vertices in the block.

\begin{theorem}[Brown--Harary,~\cite{BrownHarary1970}; Zhou--Li,~\cite{ZhouLi2023}]\label{thm:brown-harary-zhou-li}
For every $\ell\ge3$ and every $j\ge1$, if $j=(\ell-1)q'+s$ with $0\le s<\ell-1$, then
\[
\ex(j,\Cvec{\ell})
=
|A(\Fchain_{j,\ell-1})|
=
\frac12j^2+\frac{\ell-3}{2}j-
\frac{s(\ell-1-s)}{2}.
\]
The case $\ell=3$ is due to Brown and Harary, and the general directed-cycle theorem is due to Zhou and Li.
\end{theorem}

\begin{lemma}\label{lem:block-chain}
The digraph $\Fchain_{n,\ell-1}$ is $\Cvec{\ell}$-free. Moreover, its
nonincreasing out-degree sequence
$b_1\ge b_2\ge \cdots\ge b_n$
is given by
\[
b_{(i-1)(\ell-1)+a}
=
(\ell-1)(q-i+1)-1+r
\]
for $1\le i\le q$ and $1\le a\le \ell-1$, together with $r$ final terms
equal to $r-1$ if $r>0$. Finally, for every $1\le j\le n$,
\[
\sum_{i=1}^{j} b_i
=
\ex(j,\Cvec{\ell})+j(n-j).
\]
\end{lemma}

\begin{proof}
Any directed cycle using vertices from more than one block would need an
inter-block arc from a later block to an earlier block, but no such arc
exists. Hence every directed cycle lies inside a single block, whose size is
strictly smaller than $\ell$. Thus $\Fchain_{n,\ell-1}$ is
$\Cvec{\ell}$-free.

The degree formula follows directly from the construction. A vertex in the
$i$-th full block sends arcs to $\ell-2$ vertices inside its block, to
$(\ell-1)(q-i)$ vertices in later full blocks, and to $r$ vertices in the
final block. Its out-degree is therefore
\[
(\ell-2)+(\ell-1)(q-i)+r
=
(\ell-1)(q-i+1)-1+r.
\]
A vertex in the final block sends arcs only to the other $r-1$ vertices of
that block. These values are nonincreasing in the block order.

For the prefix identity, let $U_j$ be the first $j$ vertices of
$\Fchain_{n,\ell-1}$ in the block order. Then
$\Fchain_{n,\ell-1}[U_j]$ is isomorphic to $\Fchain_{j,\ell-1}$, and hence
Theorem~\ref{thm:brown-harary-zhou-li} gives
\[
|A(\Fchain_{n,\ell-1}[U_j])|=\ex(j,\Cvec{\ell}).
\]
Every arc from $U_j$ to its complement is present: this is immediate for
later blocks, and it also holds when $U_j$ cuts through a complete block.
Therefore
\[
\sum_{v\in U_j} d_{\Fchain_{n,\ell-1}}^+(v)
=
\ex(j,\Cvec{\ell})+j(n-j).
\]
Since the vertices of $U_j$ are exactly the $j$ vertices of largest
out-degree, this proves the claimed prefix identity.
\end{proof}

We now prove Theorem~\ref{thm:intro-cycles}.
\begin{proof}[Proof of Theorem~\ref{thm:intro-cycles}]
Lemma~\ref{lem:block-chain} verifies the hypothesis of
Theorem~\ref{thm:majorization} for $F=\Cvec{\ell}$ and
$G_n=\Fchain_{n,\ell-1}$. The displayed formula follows from the degree
sequence in Lemma~\ref{lem:block-chain}.
\end{proof}

\begin{corollary}\label{cor:cycle-square}
Let $\ell\ge3$, and write $n=(\ell-1)q+r$, where $0\le r<\ell-1$. Then
\[
\operatorname{ex}_2^+(n,\Cvec{\ell})
=
\sum_{i=1}^{q}
(\ell-1)\left((\ell-1)(q-i+1)-1+r\right)^2
+
\begin{cases}
r(r-1)^2, & r>0,\\
0, & r=0.
\end{cases}
\]
In particular, $\operatorname{ex}_2^+(n,\Cvec{3})=\frac{n(n^2-1)}{3}$.
\end{corollary}

\begin{proposition}\label{prop:c3-digraph-equality}
Let $D$ be a $\Cvec{3}$-free digraph on $n$ vertices. Then $\sum_{v\in V(D)} d_D^+(v)^2=\frac{n(n^2-1)}3$ if and only if $D$ is isomorphic to the ordered digon chain $\Fchain_{n,2}$.
\end{proposition}

\begin{proof}
The ordered digon chain attains equality by Corollary~\ref{cor:cycle-square}. Conversely,
let $D$ be an extremal $\Cvec{3}$-free digraph, and write its out-degrees in nonincreasing order as $d_1\ge d_2\ge\cdots\ge d_n$.
Let $b_1\ge b_2\ge \cdots\ge b_n$ be the out-degree sequence of $\Fchain_{n,2}$, namely $n-1,n-1,n-3,n-3,\ldots$, with a final $0$ when $n$ is odd. In the proof of Theorem~\ref{thm:majorization}, the sequence $(d_i)$ is weakly majorized by $(b_i)$. We claim that equality in the square-sum inequality forces
$(d_i)_{i=1}^n=(b_i)_{i=1}^n$.  Indeed, for $0\le x\le n-1$,
\[
  x^2=x+2\sum_{t=1}^{n-2}(x-t)_+ .
\]
In the proof of Theorem~\ref{thm:majorization}, each of the quantities
\[
  \sum_i d_i,\qquad
  \sum_i (d_i-t)_+ \quad (1\le t\le n-2)
\]
is bounded above by the corresponding quantity for $(b_i)$.  All
coefficients in the displayed representation of $x^2$ are positive.
Hence equality of the square sums implies equality in all these inequalities.

For a nonincreasing integer sequence $c_1\ge \cdots\ge c_n$,
\[
 \sum_i(c_i-t)_+ - \sum_i(c_i-t-1)_+
 =
 |\{i:c_i\ge t+1\}|.
\]
Applying this identity to $(d_i)$ and $(b_i)$, the equality of all hinge
sums gives
\[
  |\{i:d_i\ge s\}|=|\{i:b_i\ge s\}|
  \qquad\text{for every }s\ge 1.
\]
Therefore the two nonincreasing integer sequences are equal.

We now prove by induction on $n$ that $D\cong\Fchain_{n,2}$. The cases $n\le1$ are trivial. For $n\ge2$, the equality of degree sequences gives two vertices $u,v$ of out-degree $n-1$. Thus $uv$ and $vu$ are both arcs, and both $u$ and $v$ send arcs to every other vertex. If some vertex $w\notin\{u,v\}$ sent an arc to $u$, then the arcs $w u, u v,v w$
would form a directed triangle, a contradiction. Similarly no vertex outside $\{u,v\}$ sends an arc to $v$. Hence $\{u,v\}$ is a complete digraph block, all arcs between this block and the remaining vertices are directed away from the block, and there are no arcs in the reverse direction. Removing $u$ and $v$ leaves a $\Cvec{3}$-free digraph whose out-degree sequence is $n-3,n-3,n-5,n-5,\ldots$, with a final $0$ if $n-2$ is odd. From the induction hypothesis we know that the remaining digraph must be $\Fchain_{n-2,2}$. Therefore, we conclude that $D\cong\Fchain_{n,2}$.
\end{proof}

\section{Palettes and auxiliary digraphs}\label{sec:palettes}

We recall the palette language used in the study of uniformly dense $3$-graphs. A palette is a pair $\mathcal{P}=(\C,\A)$, where $\C$ is a finite set of colors and $\A\subseteq \C^3$ is a set of admissible ordered triples. Repetitions of colors in admissible triples are allowed. Its density is $d(\mathcal{P})=\frac{|\A|}{|\C|^3}$.
A $3$-graph $H$ is $\mathcal{P}$-colorable if there is a linear order on $V(H)$ and a coloring of all pairs of vertices by elements of $\C$ such that, whenever $x<y<z$ is an edge of $H$, the triple of pair-colors on $xy,xz,yz$ belongs to $\A$.

A homomorphism from a palette $\mathcal{P}=(\C,\A)$ to a palette $\mathcal P'=(\C',\A')$ is a map $f:\C\to\C'$ such that
\[
(a,b,c)\in\A
\quad\Longrightarrow\quad
(f(a),f(b),f(c))\in\A'.
\]
We write $\mathcal{P}\preceq \mathcal P'$ if such a homomorphism exists, and define the reverse palette of $\mathcal{P}$ by
\[
\revop(\mathcal{P})=(\C,\{(c,b,a):(a,b,c)\in\A\}).
\]

\begin{lemma}\label{lem:reverse-colorable}
A $3$-graph $H$ is $\mathcal{P}$-colorable if and only if it is $\revop(\mathcal{P})$-colorable.
\end{lemma}

\begin{proof}
Reverse the linear order on $V(H)$ and keep the same colors on unordered pairs. If an edge has color triple $(a,b,c)$ in the original order, then it has color triple $(c,b,a)$ in the reversed order.
\end{proof}

\begin{remark}\label{rem:compose-palettes}
If $\mathcal P\preceq \mathcal P'$ and $H$ is $\mathcal P$-colorable, then $H$ is $\mathcal P'$-colorable by composing the pair-coloring with the palette homomorphism.
\end{remark}

Let $D$ be a loopless digraph and let
\[
\C_D=V(D)\sqcup\{c_{uv}:uv\in A(D)\},
\]
where $c_{uv}$ are new colors. The left and right palettes generated by $D$ are
\[
\mathcal{P}_D^L=(\C_D,\{(u,v,c_{uv}):uv\in A(D)\})
\]
and
\[
\mathcal{P}_D^R=(\C_D,\{(c_{uv},u,v):uv\in A(D)\}).
\]

For a palette $\mathcal P=(\C,\A)$, define two auxiliary digraphs $G_L(\mathcal P)$ and $G_R(\mathcal P)$ on vertex set $\C$ by
\[
xy\in A(G_L(\mathcal P))
\quad\Longleftrightarrow\quad
\exists z\in\C\text{ such that }(x,y,z)\in\A,
\]
and
\[
yz\in A(G_R(\mathcal P))
\quad\Longleftrightarrow\quad
\exists x\in\C\text{ such that }(x,y,z)\in\A.
\]
The auxiliary digraphs may have loops.

\begin{lemma}\label{lem:palette-auxiliary}
Let $T$ be a tournament and $\mathcal P$ a palette. Then
\[
\mathcal{P}_T^L\preceq \mathcal P
\quad\Longleftrightarrow\quad
T\to G_L(\mathcal P),
\]
and
\[
\mathcal{P}_T^R\preceq \mathcal P
\quad\Longleftrightarrow\quad
T\to G_R(\mathcal P).
\]
\end{lemma}

\begin{proof}
Write $\mathcal P=(\C,\A)$. We prove the first equivalence; the second is identical. Suppose first that $f:\mathcal{P}_T^L\to \mathcal P$ is a palette homomorphism. Then the restriction of $f$ to $V(T)$ maps every arc $uv\in A(T)$ to an arc $f(u)f(v)\in A(G_L(\mathcal P))$, because $(f(u),f(v),f(c_{uv}))\in\A$.
Thus $f|_{V(T)}$ is a digraph homomorphism from $T$ to $G_L(\mathcal P)$.

Conversely, suppose that $g:V(T)\to \C$ is a digraph homomorphism from $T$ to $G_L(\mathcal P)$. For every arc $uv\in A(T)$, choose a color $w_{uv}\in \C$ such that
$(g(u),g(v),w_{uv})\in\A$.
Extend $g$ to $V(T)\cup A(T)$ by sending $c_{uv}$ to $w_{uv}$. This extension is a palette homomorphism from $\mathcal{P}_T^L$ to $\mathcal P$.
\end{proof}

\begin{remark}\label{rem:tournament-injective}
Every homomorphism from a tournament to a loopless digraph is injective.
\end{remark}

\section{Self-converse tournaments and finite palette extremality}\label{sec:self-converse}

Let $T$ be a tournament. Using the notation $\C(\mathcal P)$ and $\A(\mathcal P)$ for the color set and admissible-triple set of a palette $\mathcal P$, define
\[
\pal_T(m)=
\max\{|\A(\mathcal P)|: |\C(\mathcal P)|=m,\; \mathcal{P}_T^L\npreceq \mathcal P,\; \mathcal{P}_T^R\npreceq \mathcal P\}.
\]
We now prove Theorem~\ref{thm:intro-palette}.
\begin{proof}[Proof of Theorem~\ref{thm:intro-palette}]
We prove the two inequalities $\pal_T(m)\leq \operatorname{ex}_2^+(m,T)$ and $\pal_T(m)\geq \operatorname{ex}_2^+(m,T)$.
Let $\mathcal P=(\C,\A)$ be an $m$-color palette such that $\mathcal{P}_T^L\npreceq \mathcal P$ and $\mathcal{P}_T^R\npreceq \mathcal P$. We first show that $G_L(\mathcal P)$ and $G_R(\mathcal P)$ are loopless. If $G_L(\mathcal P)$ has a loop at $x$, then there exists $y\in\C$ such that $(x,x,y)\in\A$. Since $T$ has at least one arc, mapping every vertex of $T$ to $x$ and every arc-color $c_{uv}$ to $y$ gives a palette homomorphism $\mathcal{P}_T^L\to \mathcal P$, a contradiction. The proof for $G_R(\mathcal P)$ is analogous.

By Lemma~\ref{lem:palette-auxiliary}, neither $G_L(\mathcal P)$ nor $G_R(\mathcal P)$ admits a homomorphism from $T$. In particular, both auxiliary digraphs are $T$-free. For each $y\in\C$, let $B_y=\{(x,z)\in\C^2:(x,y,z)\in\A\}$.
If $(x,z)\in B_y$, then $xy\in A(G_L(\mathcal P))$ and $yz\in A(G_R(\mathcal P))$. Hence $|B_y|\le d^-_{G_L(\mathcal P)}(y)d^+_{G_R(\mathcal P)}(y)$.
It follows from Cauchy--Schwarz that
\[
\begin{aligned}
|\A|
&=\sum_{y\in\C}|B_y| \\
&\le
\sum_{y\in\C}d^-_{G_L(\mathcal P)}(y)d^+_{G_R(\mathcal P)}(y) \\
&\le
\left(\sum_{y\in\C} d^-_{G_L(\mathcal P)}(y)^2\right)^{1/2}
\left(\sum_{y\in\C} d^+_{G_R(\mathcal P)}(y)^2\right)^{1/2}.
\end{aligned}
\]
Since $G_R(\mathcal P)$ is $T$-free, $\sum_y d^+_{G_R(\mathcal P)}(y)^2\le \operatorname{ex}_2^+(m,T)$.
Moreover, $G_L(\mathcal P)^{\mathrm{rev}}$ is also $T$-free: otherwise $T\to G_L(\mathcal P)^{\mathrm{rev}}$ would imply $T^{\mathrm{rev}}\to G_L(\mathcal P)$, and since $T$ is self-converse this would give $T\to G_L(\mathcal P)$. Therefore
\[
\sum_y d^-_{G_L(\mathcal P)}(y)^2
=
\sum_y d^+_{G_L(\mathcal P)^{\mathrm{rev}}}(y)^2
\le
\operatorname{ex}_2^+(m,T).
\]
Thus $|\A|\le \operatorname{ex}_2^+(m,T)$.

For the lower bound, let $D$ be a $T$-free digraph on $m$ vertices with $\sum_{y\in V(D)} d_D^+(y)^2=\operatorname{ex}_2^+(m,T)$.
Define a palette $\mathcal{P}_D=(V(D),\A_D)$ by
\[
(x,y,z)\in\A_D
\quad\Longleftrightarrow\quad
yx\in A(D)\text{ and }yz\in A(D).
\]
Then $|\A_D|=\sum_{y\in V(D)}d_D^+(y)^2$.
Repetitions of colors in triples are allowed. Hence, if $yz\in A(D)$, then taking $x=z$ shows $yz\in A(G_R(\mathcal{P}_D))$, and similarly for $G_L(\mathcal{P}_D)$. Therefore
$G_R(\mathcal{P}_D)=D$ and $G_L(\mathcal{P}_D)=D^{\mathrm{rev}}$.

Since $D$ is $T$-free and $T$ is self-converse, both auxiliary digraphs are $T$-free. Lemma~\ref{lem:palette-auxiliary} gives $\mathcal{P}_T^L\npreceq \mathcal{P}_D$ and $\mathcal{P}_T^R\npreceq \mathcal{P}_D$. Thus $\pal_T(m)\ge \operatorname{ex}_2^+(m,T)$, completing the proof.
\end{proof}

\subsection{The transitive triangle}\label{subsec:transitive-triangle}

The transitive tournament $TT_3$ on three vertices is self-converse, yet its
behavior differs significantly from that of the directed triangle.

\begin{theorem}\label{thm:transitive-triangle-square}
For every $m\ge1$, $\operatorname{ex}_2^+(m,TT_3)
=m\left\lfloor\frac{m^2}{4}\right\rfloor$.
Moreover, equality is attained precisely by the complete bipartite digraph with part sizes $\lfloor m/2\rfloor$ and $\lceil m/2\rceil$.
\end{theorem}

\begin{proof}
The cases $m\le 2$ are immediate, so it is sufficient to consider $m\ge 3$. Let $D$ be a
$TT_3$-free digraph on $m$ vertices. For every vertex $w$ of $D$, the
out-neighborhood $N_D^+(w)$ is independent in the underlying graph of
$D$: indeed, if $x,y\in N_D^+(w)$ and at least one of $xy$ or $yx$
is an arc, then $w,x,y$ span a transitive triangle.

Choose a vertex $x_0$ of maximum out-degree, and write $\Delta:=d_D^+(x_0)$,
$X:=N_D^+(x_0)$ and $Y:=V(D)\setminus X$.
Then $|X|=\Delta$ and $|Y|=m-\Delta$. By the observation above, there
are no arcs inside $X$. Hence every vertex of $X$ has all its
out-neighbors in $Y$, and therefore has out-degree at most
$|Y|=m-\Delta$. Every vertex of $Y$ has out-degree at most $\Delta$,
by the choice of $x_0$. Consequently,
\[
\sum_{u\in V(D)} d_D^+(u)^2
\le
|Y|\Delta^2+|X|(m-\Delta)^2
=
(m-\Delta)\Delta^2+\Delta(m-\Delta)^2
=
m\Delta(m-\Delta)
\le
m\left\lfloor \frac{m^2}{4}\right\rfloor.
\]

For the lower bound, take a partition $V=A\cup B$ with $|A|=\left\lfloor \frac m2\right\rfloor$ and $|B|=\left\lceil \frac m2\right\rceil$,
put both arcs between every pair $a\in A$, $b\in B$, and put no arcs
inside the parts. This digraph is $TT_3$-free, since every
out-neighborhood is one of the two independent parts. Its out-degree
square sum is
\[
|A||B|^2+|B||A|^2
=
m|A||B|
=
m\left\lfloor\frac{m^2}{4}\right\rfloor.
\]

It remains to prove the equality statement. Suppose that $D$ is
extremal. Then equality must hold in each of the inequalities used above: in the
bounds on the out-degrees of vertices in $X$ and $Y$, and in the
inequality $\Delta(m-\Delta)\le \lfloor m^2/4\rfloor$. In
particular, $\Delta(m-\Delta)=\left\lfloor\frac{m^2}{4}\right\rfloor$.
Moreover, every vertex of $X$ must have out-degree $m-\Delta$, and
every vertex of $Y$ must have out-degree $\Delta$. Since $\Delta$
is the maximum out-degree, the first condition implies $m-\Delta\le \Delta$.
Together with the product equality above, this forces
$\Delta=\left\lceil\frac m2\right\rceil$ and $|Y|=m-\Delta=\left\lfloor\frac m2\right\rfloor$.

Since $X$ is independent and every vertex of $X$ has out-degree
$|Y|$, every vertex of $X$ sends an arc to every vertex of $Y$.

We now show that $Y$ is independent. Suppose, on the contrary, that
there exist $y,y'\in Y$ with $y\to y'$. If $y$ sent an arc to some
$x\in X$, then, since every vertex of $X$ sends an arc to every
vertex of $Y$, we would have $y\to x, y\to y', x\to y'$, and hence $y,x,y'$ would span a transitive triangle, a contradiction.
Thus $y$ sends no arcs to $X$. Consequently, $d_D^+(y)\le |Y|-1<\Delta$,
contradicting the equality condition that every vertex of $Y$ has
out-degree $\Delta$. This proves that $Y$ is independent.

Since $Y$ is independent and every vertex of $Y$ has out-degree
$\Delta=|X|$, every vertex of $Y$ sends an arc to every vertex of
$X$. We have proved that both arcs are present between every vertex of
$X$ and every vertex of $Y$, and that there are no arcs inside either
part. Hence $D$ is the complete bipartite digraph with
balanced part sizes.
\end{proof}

\begin{corollary}\label{cor:transitive-triangle-palette}
For every $m\ge1$, $\pal_{TT_3}(m)=m\left\lfloor\frac{m^2}{4}\right\rfloor$.
Equivalently, the sharp $m$-color palette density for avoiding both the left and right transitive-triangle palettes is $\frac{1}{m^2}\left\lfloor\frac{m^2}{4}\right\rfloor$.
\end{corollary}

\begin{proof}
Since $TT_3$ is self-converse, this corollary follows immediately from Theorems~\ref{thm:intro-palette} and~\ref{thm:transitive-triangle-square}.
\end{proof}

\section{The finite endpoint below one third}\label{sec:endpoint}

Let $\mathcal{P}_L=\mathcal{P}_{\Cvec{3}}^L$ and $\mathcal{P}_R=\mathcal{P}_{\Cvec{3}}^R$.
Since $\Cvec{3}$ is self-converse, $\mathcal{P}_L$ is isomorphic to $\revop(\mathcal{P}_R)$, and $\mathcal{P}_R$ is isomorphic to $\revop(\mathcal{P}_L)$.

We first recall two useful results that will be used in the proof.
The first is Lamaison's palette characterization of uniform Tur\'an density. 

\begin{theorem}[Lamaison,~\cite{Lamaison2024}]\label{thm:lamaison}
For every finite $3$-graph $H$, \[\piu(H)=\sup\{d(\mathcal P): H\text{ is not }\mathcal P\text{-colorable}\}.\]
\end{theorem}

The second is the palette separation theorem of Kr\'al', Ku\v{c}er\'ak, Lamaison and Tardos.

\begin{theorem}[{Kr\'al'--Ku\v{c}er\'ak--Lamaison--Tardos~\cite[Theorem~13]{KKLT2025}}]\label{thm:separation}
Let $\mathcal P$ and $\mathcal{P}_0$ be palettes. There exists a finite $3$-graph that is $\mathcal P$-colorable but not $\mathcal{P}_0$-colorable if and only if $\mathcal P\npreceq \mathcal{P}_0$ and $\mathcal P\npreceq \revop(\mathcal{P}_0)$.
\end{theorem}

\begin{corollary}\label{cor:triangle-palette-exact}
For every $m\ge1$, $\pal_{\Cvec{3}}(m) =\frac{m(m^2-1)}{3}$.
Equivalently, if $\mathcal P=(\C,\A)$ is an $m$-color palette such that $\mathcal{P}_L\npreceq \mathcal P$ and $\mathcal{P}_R\npreceq \mathcal P$, then $d(\mathcal P)\le \frac13-\frac1{3m^2}$.
This bound is best possible.
\end{corollary}

\begin{proof}
The identity follows from Theorem~\ref{thm:intro-palette} and Corollary~\ref{cor:cycle-square}. Dividing by $m^3$ gives the density formulation, and the lower-bound construction in Theorem~\ref{thm:intro-palette} shows sharpness.
\end{proof}
An \emph{ordered digon chain} on an $m$-element set is a copy of
$\Fchain_{m,2}$ obtained from some ordering of that set. For an ordered digon
chain $J$ on a color set $\C$, define the endpoint palette
$\mathcal P_m^\star(J)$ by
\[
(x,y,z)\in\A(\mathcal P_m^\star(J))
\quad\Longleftrightarrow\quad
y\to x\text{ and }y\to z\text{ in }J.
\]
When $J=\Fchain_{m,2}$ carries its standard ordering, we write simply
$\mathcal P_m^\star$. The lower-bound construction in
Theorem~\ref{thm:intro-palette} gives
\[
G_R(\mathcal P_m^\star(J))=J,
\qquad
G_L(\mathcal P_m^\star(J))=J^{\mathrm{rev}},
\qquad
|\A(\mathcal P_m^\star(J))|=\sum_{y\in\C}d_J^+(y)^2.
\]
The next result shows the equality conditions behind Corollary~\ref{cor:triangle-palette-exact}. It is often more useful than a purely numerical statement because it describes how the two auxiliary digraphs must align.

\begin{theorem}\label{thm:triangle-palette-equality}
Let $\mathcal P=(\C,\A)$ be an $m$-color palette such that $\mathcal{P}_L\npreceq \mathcal P$ and $\mathcal{P}_R\npreceq \mathcal P$. For each $y\in\C$, write
\[
B_y=\{(x,z)\in\C^2:(x,y,z)\in\A\}.
\]
Then
$|\A|=\frac{m(m^2-1)}3$
if and only if the following three conditions hold:
\begin{enumerate}[label=\textup{(E\arabic*)}]
\item $G_R(\mathcal{P})\cong \Fchain_{m,2}$ and $G_L(\mathcal{P})^{\mathrm{rev}}\cong \Fchain_{m,2}$;
\item for every $y\in\C$, $d^-_{G_L(\mathcal{P})}(y)=d^+_{G_R(\mathcal{P})}(y)$;
\item for every $y\in\C$, $B_y=N^-_{G_L(\mathcal{P})}(y)\times N^+_{G_R(\mathcal{P})}(y)$.
\end{enumerate}
\end{theorem}

\begin{proof}
Assume first that $|\A|=m(m^2-1)/3$.
We recall the following inequalities from the proof of Theorem~\ref{thm:intro-palette} with $T=\Cvec{3}$:
\begin{align*}
|\A|
&\le \sum_{y\in\C}d^-_{G_L(\mathcal{P})}(y)d^+_{G_R(\mathcal{P})}(y) \\
&\le
\left(\sum_{y\in\C} d^-_{G_L(\mathcal{P})}(y)^2\right)^{1/2}
\left(\sum_{y\in\C} d^+_{G_R(\mathcal{P})}(y)^2\right)^{1/2} \\
&\le \frac{m(m^2-1)}3.
\end{align*}

As $|\A|=m(m^2-1)/3$, all inequalities above must be equalities.
In particular,
\[
\sum_y d^+_{G_R(\mathcal{P})}(y)^2
=\sum_y d^-_{G_L(\mathcal{P})}(y)^2
=\frac{m(m^2-1)}3.
\]
The auxiliary digraphs $G_L(\mathcal P)$ and $G_R(\mathcal P)$ are loopless and
$\Cvec{3}$-free; since $\Cvec{3}$ is self-converse,
$G_L(\mathcal P)^{\mathrm{rev}}$ is $\Cvec{3}$-free as well, and $d^-_{G_L(\mathcal P)}(y)=d^+_{G_L(\mathcal P)^{\mathrm{rev}}}(y)$ for every $y\in\C$. Hence
Proposition~\ref{prop:c3-digraph-equality}, applied to $G_R(\mathcal P)$ and to $G_L(\mathcal P)^{\mathrm{rev}}$, implies
$G_R(\mathcal{P})\cong\Fchain_{m,2}$ and
$G_L(\mathcal{P})^{\mathrm{rev}}\cong\Fchain_{m,2}$.
Equality in Cauchy--Schwarz gives proportionality of the two nonnegative vectors
$\big(d^-_{G_L(\mathcal{P})}(y)\big)_{y\in\C}$ and $\big(d^+_{G_R(\mathcal{P})}(y)\big)_{y\in\C}$.
Their squared norms are equal, so the proportionality constant is $1$, proving (E2). Finally, for each $y$, we always have
$B_y\subseteq N^-_{G_L(\mathcal{P})}(y)\times N^+_{G_R(\mathcal{P})}(y)$.
Since equality also holds in the first inequality, all these inclusions must be equalities, proving (E3).

Conversely, suppose (E1)--(E3) hold. Then
\[
\begin{aligned}
|\A|
&=\sum_{y\in\C}|B_y| \\
&=\sum_{y\in\C}d^-_{G_L(\mathcal{P})}(y)d^+_{G_R(\mathcal{P})}(y) \\
&=\sum_{y\in\C}d^+_{G_R(\mathcal{P})}(y)^2
=\frac{m(m^2-1)}3,
\end{aligned}
\]
where the last equality follows from (E1) and Corollary~\ref{cor:cycle-square}. This proves the theorem.
\end{proof}
\begin{corollary}\label{cor:extremal-palettes}
Let $\mathcal P=(\C,\A)$ be an $m$-color palette satisfying
$\mathcal P_L\npreceq\mathcal P$ and $\mathcal P_R\npreceq\mathcal P$.
Then $|\A|=\frac{m(m^2-1)}3$
if and only if $\mathcal P=\mathcal P_m^\star(J)$ for some ordered digon
chain $J$ on $\C$. Consequently, up to a relabeling of the colors,
$\mathcal P_m^\star$ is the unique extremal palette.
\end{corollary}

\begin{proof}
The reverse implication follows from Corollary~\ref{cor:cycle-square} and the
construction above. For the forward implication, let
$J_R=G_R(\mathcal P)$ and
$J_L=G_L(\mathcal P)^{\mathrm{rev}}$.
By Theorem~\ref{thm:triangle-palette-equality}, both $J_L$ and $J_R$ are
ordered digon chains and
$d_{J_L}^+(y)=d_{J_R}^+(y)$ for every $y\in\C$. In an ordered digon chain on $m$ vertices, the possible out-degrees are
$m-1,m-1,m-3,m-3,\ldots$ with, in the odd case, a final value $0$.  Thus the out-degree determines
the block, or degree layer, of the vertex.  Once the ordered partition into
these layers is known, all arcs are determined: every pair inside a two-vertex
layer forms a digon, and every pair in two distinct layers is oriented from
the earlier layer to the later layer. Hence $J_L=J_R=:J$. Condition~\textup{(E3)} now gives
$B_y=N_J^+(y)\times N_J^+(y)$
for every $y\in\C$,
which is precisely the identity $\mathcal P=\mathcal P_m^\star(J)$.
\end{proof}

\begin{corollary}\label{cor:sharp-contrapositive}
Let $\mathcal P=(\C,\A)$ be an $m$-color palette. If $d(\mathcal P)>\frac13-\frac1{3m^2}$,
then $\mathcal{P}_L\preceq \mathcal P$ or $\mathcal{P}_R\preceq \mathcal P$.
\end{corollary}

The equality analysis admits a quantitative refinement. The following exact defect identity decomposes the gap below the endpoint into four nonnegative contributions.

\begin{proposition}\label{prop:defect}
Let $\mathcal P=(\C,\A)$ be an $m$-color palette with $\mathcal{P}_L\npreceq \mathcal P$ and $\mathcal{P}_R\npreceq \mathcal P$. Let
\[
E_m=\frac{m(m^2-1)}3,\qquad
a_y=d^-_{G_L(\mathcal{P})}(y),\qquad
b_y=d^+_{G_R(\mathcal{P})}(y),
\]
and $t_y=|B_y|=|\{(x,z):(x,y,z)\in\A\}|$. Then
\[
E_m-|\A|
=\frac12\!\left(E_m-\sum_y a_y^2\right)
+\frac12\!\left(E_m-\sum_y b_y^2\right)
+\frac12\sum_y(a_y-b_y)^2
+\sum_y(a_yb_y-t_y).
\]
Every term on the right-hand side is nonnegative.
\end{proposition}

\begin{proof}
Since $\mathcal P$ avoids $\mathcal{P}_L$ and $\mathcal{P}_R$, the auxiliary digraphs $G_L(\mathcal{P})$ and $G_R(\mathcal{P})$ are loopless and $\Cvec{3}$-free; since $\Cvec{3}$ is self-converse,
$G_L(\mathcal P)^{\mathrm{rev}}$ is $\Cvec{3}$-free as well, and
$a_y=d^+_{G_L(\mathcal P)^{\mathrm{rev}}}(y)$ for every $y\in\C$. Hence, by
Corollary~\ref{cor:cycle-square}, applied to $G_L(\mathcal P)^{\mathrm{rev}}$
and to $G_R(\mathcal P)$,
we have $\sum_y a_y^2\le E_m$ and $\sum_y b_y^2\le E_m$.
Also $t_y\le a_yb_y$ for every middle color $y$, because $B_y\subseteq N^-_{G_L(\mathcal{P})}(y)\times N^+_{G_R(\mathcal{P})}(y)$. Thus all displayed summands are nonnegative. Finally, since $\sum_y t_y=|\A|$,
\[
\begin{aligned}
&\frac12\!\left(E_m-\sum_y a_y^2\right)
+\frac12\!\left(E_m-\sum_y b_y^2\right)
+\frac12\sum_y(a_y-b_y)^2
+\sum_y(a_yb_y-t_y)\\
&\qquad=E_m-\sum_y a_yb_y+\sum_y(a_yb_y-t_y)
=E_m-\sum_y t_y=E_m-|\A|.
\end{aligned}
\]
\end{proof}

We now prove Theorem~\ref{thm:intro-limit}.
\begin{proof}[Proof of Theorem~\ref{thm:intro-limit}]
Let $\mathcal P_m^\star$ be the standard endpoint palette defined above. By the construction above and Corollary~\ref{cor:cycle-square}, this palette avoids
both $\mathcal P_L$ and $\mathcal P_R$ and satisfies
$d(\mathcal P_m^\star)=\frac13-\frac1{3m^2}$.
Since $\revop(\mathcal{P}_R)\cong \mathcal{P}_L$ and $\mathcal{P}_L\npreceq \mathcal{P}_m^\star$, we have $\mathcal{P}_R\npreceq \revop(\mathcal{P}_m^\star)$ which together with $\mathcal{P}_R\npreceq \mathcal{P}_m^\star$ verifies the two non-homomorphism conditions needed below.

By Theorem~\ref{thm:separation}, there exists a finite $3$-graph $H_m$ that is $\mathcal{P}_R$-colorable but not $\mathcal{P}_m^\star$-colorable. Theorem~\ref{thm:lamaison} then gives $\piu(H_m)\ge d(\mathcal{P}_m^\star)=\frac13-\frac1{3m^2}$.

For the upper bound, let $\mathcal Q=(\C_Q,\A_Q)$ be any finite palette with $d(\mathcal Q)>1/3$, and put $s=|\C_Q|$. Then
$d(\mathcal Q)>\frac13-\frac1{3s^2}$.
By Corollary~\ref{cor:sharp-contrapositive}, it holds that either $\mathcal{P}_L\preceq \mathcal Q$ or $\mathcal{P}_R\preceq \mathcal Q$. Since $H_m$ is $\mathcal{P}_R$-colorable, Lemma~\ref{lem:reverse-colorable} and the isomorphism $\revop(\mathcal{P}_R)\cong \mathcal{P}_L$ imply that $H_m$ is also $\mathcal{P}_L$-colorable. Hence, in either case, Remark~\ref{rem:compose-palettes} shows that $H_m$ is $\mathcal Q$-colorable. Thus every palette of density greater than $1/3$ colors $H_m$. By Theorem~\ref{thm:lamaison}, $\piu(H_m)\le1/3$.
\end{proof}

\section{Edit-distance stability}\label{sec:edit-dist}

The defect identity above is an exact decomposition of the endpoint deficit.
We first establish a finite stability form of the $\Cvec{3}$ degree-square
theorem and then transfer it to palettes. Throughout this section, write
\[
E_m=\frac{m(m^2-1)}3=\operatorname{ex}_2^+(m,\Cvec{3}).
\]

\begin{lemma}\label{lem:continuous-majorization}
Let $f,g:[0,1]\to[0,1]$ be nonincreasing measurable functions such that
\[
\int_0^s f(x)\,dx\le \int_0^s g(x)\,dx
\qquad\text{for every }s\in[0,1].
\]
Then
\[
\int_0^1 f(x)^2\,dx\le \int_0^1 g(x)^2\,dx.
\]
If equality holds, then $f=g$ almost everywhere.
\end{lemma}

\begin{proof}
For a nonincreasing function $h:[0,1]\to[0,1]$ and $t\in[0,1]$,
\[
\int_0^1(h(x)-t)_+\,dx
=
\max_{0\le s\le1}\left(\int_0^s h(x)\,dx-st\right).
\]
Indeed, the maximum is attained by an initial segment on which $h$ is
largest. The assumed prefix inequalities therefore imply
\[
\int_0^1(f(x)-t)_+\,dx
\le
\int_0^1(g(x)-t)_+\,dx
\qquad(0\le t\le1).
\]
Since
\[
u^2=2\int_0^1(u-t)_+\,dt
\qquad(0\le u\le1),
\]
integrating the preceding inequality over $t$ gives
\[
\int_0^1 f(x)^2\,dx\le \int_0^1 g(x)^2\,dx.
\]

If equality holds, then the two hinge integrals agree for almost every
$t\in[0,1]$. They are continuous functions of $t$, and hence they agree for
every $t\in[0,1]$. For almost every $t$, their derivatives are
$-\lambda(\{x:f(x)>t\})$ and
$-\lambda(\{x:g(x)>t\})$,
respectively. Thus $f$ and $g$ have the same distribution function. Since
both functions are nonincreasing, they are equal almost everywhere.
\end{proof}
For a loopless digraph $D$ on $m$ vertices, encode each unordered pair
$\{u,v\}$ by one of the four states $00,10,01,11$ according to the presence
of the arcs $uv$ and $vu$; the \emph{associated colored graphon} of $D$ is
the $4$-tuple of $\{0,1\}$-valued step kernels
$(W^D_{00},W^D_{10},W^D_{01},W^D_{11})$ on vertex cells of measure $1/m$
recording these states, which sum to $1$ off the diagonal cells and satisfy
$W^D_{ab}(x,y)=W^D_{ba}(y,x)$. The colored cut metric applies the cut
distance coordinatewise with one common measure-preserving map. Throughout,
$\lambda$ denotes Lebesgue measure. For a $4$-tuple $\mathbf W$ of kernels,
write $W^+=W_{10}+W_{11}$ and
$e^+(S,T)=\int_{S\times T}W^+(x,y)\,dx\,dy$ for measurable
$S,T\subseteq[0,1]$; for a digraph $D$, let $e^+_D(S,T)$ denote the number of
arcs of $D$ from $S$ to $T$.
\begin{lemma}\label{lem:rounding}
Let $(D_i)$ be loopless digraphs with $|V(D_i)|=m_i\to\infty$ whose associated
colored graphons converge to $W=(W_{00},W_{10},W_{01},W_{11})$ in the colored
cut metric, and let $W^+=W_{10}+W_{11}$. Let $U_1,\dots,U_K$ be a fixed
measurable partition of $[0,1]$. Then there are partitions
$V(D_i)=V_{i,1}\cup\dots\cup V_{i,K}$ such that, as $i\to\infty$,
\[
|V_{i,p}|=\lambda(U_p)\,m_i+o(m_i)
\quad\text{and}\quad
e^+_{D_i}(V_{i,p},V_{i,q})=e^+(U_p,U_q)\,m_i^2+o(m_i^2)
\qquad(1\le p,q\le K).
\]
\end{lemma}

\begin{proof}
There are measure-preserving bijections $\psi_i$ of $[0,1]$ with
$\|(W^{D_i}_{ab})^{\psi_i}-W_{ab}\|_\square\to 0$ for each state $ab$, where
$U^{\psi}(x,y)=U(\psi(x),\psi(y))$; that the colored cut distance can be
computed along measure-preserving bijections is standard, see
\cite[Section~8.2]{LovaszBook2012} and \cite{BCLSV}. Writing $W_i^+$ for the arc kernel of $D_i$ and
summing the states $10$ and $11$ gives
$\|(W_i^+)^{\psi_i}-W^+\|_\square\to 0$.

Let $C_{i,p}=\psi_i(U_p)$, a measurable partition of $[0,1]$ with
$\lambda(C_{i,p})=\lambda(U_p)$. Testing the cut norm on $U_p\times U_q$ and
changing variables,
\[
\int_{C_{i,p}\times C_{i,q}}W_i^+
=\int_{U_p\times U_q}(W_i^+)^{\psi_i}
\;\longrightarrow\; e^+(U_p,U_q)
\qquad(1\le p,q\le K).
\tag{$\ast$}
\]

For a vertex $v$ of $D_i$ with cell $I_v\subseteq[0,1]$, set
$x_{v,p}=m_i\,\lambda(I_v\cap C_{i,p})$, so that $x_{v,p}\in[0,1]$ and
$\sum_p x_{v,p}=1$. Assign each vertex independently to a random part
$\pi(v)$ with $\Pr(\pi(v)=p)=x_{v,p}$, and let $V_{i,p}=\pi^{-1}(p)$. Then
$\mathbb{E}|V_{i,p}|=m_i\lambda(C_{i,p})=\lambda(U_p)m_i$. Moreover, since
$D_i$ is loopless, $W_i^+$ vanishes on the diagonal cells, so
\[
\int_{C_{i,p}\times C_{i,q}}W_i^+
=\frac{1}{m_i^2}\sum_{v\ne w}a_{vw}\,x_{v,p}\,x_{w,q},
\]
where $a_{vw}$ is the indicator of the arc $vw$, while independence of the
assignments gives
\[
\mathbb{E}\,e^+_{D_i}(V_{i,p},V_{i,q})=\sum_{v\ne w}a_{vw}\,x_{v,p}\,x_{w,q}.
\]
Hence $\mathbb{E}\,e^+_{D_i}(V_{i,p},V_{i,q})
=m_i^2\int_{C_{i,p}\times C_{i,q}}W_i^+$, which by $(\ast)$ equals
$e^+(U_p,U_q)\,m_i^2+o(m_i^2)$.

Finally, reassigning one vertex changes each $|V_{i,p}|$ by at most $1$ and
each $e^+_{D_i}(V_{i,p},V_{i,q})$ by at most $2m_i$. By the
bounded-difference inequality, with probability $1-o(1)$ every $|V_{i,p}|$
deviates from its expectation by at most $m_i^{3/4}$ and every
$e^+_{D_i}(V_{i,p},V_{i,q})$ by at most $m_i^{7/4}$. Fixing such an
assignment completes the proof.
\end{proof}
\begin{lemma}\label{lem:C3-digraph-stability}
For every $\rho>0$ there exists $\theta>0$ such that the following holds for
every $m$. If $D$ is a loopless $\Cvec{3}$-free digraph on $m$ vertices and
$\sum_{v\in V(D)} d_D^+(v)^2\ge E_m-\theta m^3$,
then there is an ordered digon chain $J_m$ on the same vertex set such that
$|A(D)\triangle A(J_m)|\le \rho m^2$.
\end{lemma}

\begin{proof}
Suppose that the statement fails for some $\rho_0>0$. For every positive
integer $i$, choose a loopless $\Cvec{3}$-free digraph $D_i$ on $m_i$
vertices such that
\[
\sum_{v\in V(D_i)}d_{D_i}^+(v)^2\ge E_{m_i}-\frac{m_i^3}{i},
\]
but every ordered digon chain on $V(D_i)$ differs from $D_i$ in more than
$\rho_0m_i^2$ arcs.

We may assume that $m_i\to\infty$. Indeed, if $(m_i)$ were bounded, then
after passing to a subsequence we would have $m_i=m$ fixed. Since
$\sum_{v\in V(D_i)}d_{D_i}^+(v)^2$
is an integer and $m^3/i<1$ for all sufficiently large $i$, the inequality
above would imply
$\sum_{v\in V(D_i)}d_{D_i}^+(v)^2=E_m$ for all sufficiently large $i$. Proposition~\ref{prop:c3-digraph-equality}
would then imply that $D_i$ is an ordered digon chain, contradicting the
choice of $D_i$.

Encode each $D_i$ by its associated colored graphon as above. By compactness for finite edge-colored graphons (see
\cite[Theorem~9.23]{LovaszBook2012} for the one-kernel statement, whose proof
applies verbatim to tuples of kernels summing to $1$; cf.\
\cite{LovaszSzegedy2006}), after passing to a subsequence
and choosing suitable representatives, the corresponding four-colored
graphons converge in the colored cut metric to
$\mathbf W=(W_{00},W_{10},W_{01},W_{11}).$
Here the functions $W_{ab}:[0,1]^2\to[0,1]$ are measurable, sum to $1$
almost everywhere, and satisfy
$W_{ab}(x,y)=W_{ba}(y,x)$.
Define the marginal arc kernel
\[
W^+(x,y)=W_{10}(x,y)+W_{11}(x,y),
\]
and its out-degree function
\[
d_W^+(x)=\int_0^1W^+(x,y)\,dy.
\]

For every fixed oriented digraph $F$, meaning that $F$ has at most one arc
on each unordered pair, define
\[
t(F,\mathbf W)
=
\int_{[0,1]^{V(F)}}
\prod_{uv\in A(F)}W^+(x_u,x_v)\,
\prod_{u\in V(F)}dx_u.
\]
This density is a finite sum of colored homomorphism densities and is
therefore continuous along the convergent sequence. We use this only for the
out-star $\overrightarrow S_2$ and the directed triangle $\Cvec{3}$.

For every finite loopless digraph $G$ on $n$ vertices,
\[
t(\overrightarrow S_2,G)
=
\frac1{n^3}\sum_{v\in V(G)}d_G^+(v)^2,
\]
while
\[
t(\overrightarrow S_2,\mathbf W)
=
\int_0^1 d_W^+(x)^2\,dx.
\]
Consequently,
\[
\int_0^1d_W^+(x)^2\,dx
=
\lim_{i\to\infty}
\frac1{m_i^3}
\sum_{v\in V(D_i)}d_{D_i}^+(v)^2
=
\frac13.
\tag{7.1}\label{eq:limit-square}
\]
Moreover, since every $D_i$ is $\Cvec{3}$-free, we have
$t(\Cvec{3},D_i)=0$ for every $i$,
and hence
\[
t(\Cvec{3},\mathbf W)=0.
\tag{7.2}\label{eq:limit-triangle-free}
\]

We first prove the graphon form of the Brown--Harary bound:
\[
e^+(S,S)\le\frac{\lambda(S)^2}{2}
\qquad\text{for every measurable }S\subseteq[0,1].
\tag{7.3}\label{eq:graphon-BH}
\]
The claim is trivial when $\lambda(S)=0$. Suppose $\lambda(S)>0$. Since the
integrand defining $t(\Cvec{3},\mathbf W)$ is nonnegative, the identity
\eqref{eq:limit-triangle-free} implies
\[
\int_{S^3}
W^+(x,y)W^+(y,z)W^+(z,x)\,dx\,dy\,dz=0.
\tag{7.4}\label{eq:restricted-triangle-zero}
\]
Thus the rescaled restriction of $\mathbf W$ to $S$ also has zero directed
triangle density.

Assume for contradiction that the rescaled restriction has arc density
$\alpha:=\lambda(S)^{-2}\,e^+(S,S)>1/2$. Sample $n$ latent points independently from the rescaled
space and, for each unordered pair, independently sample one of the four arc
states according to the restricted colored graphon. The sampled digraph is
loopless. By \eqref{eq:restricted-triangle-zero}, its expected number of
labeled directed triangles is zero. Since this number is nonnegative, the
sampled digraph is $\Cvec{3}$-free almost surely.

On the other hand, the expected number of arcs in the sampled digraph is
$\alpha n(n-1)$. For all sufficiently large $n$, this is larger than
$\operatorname{ex}(n,\Cvec{3})$, by
Theorem~\ref{thm:brown-harary-zhou-li}. This is impossible, because every
$\Cvec{3}$-free outcome has at most $\operatorname{ex}(n,\Cvec{3})$ arcs.
Therefore $\alpha\le1/2$, proving \eqref{eq:graphon-BH}.

Let $S\subseteq[0,1]$ have measure $s$. Then
\[
\int_Sd_W^+(x)\,dx
=
e^+(S,S)+e^+(S,S^c)
\le
\frac{s^2}{2}+s(1-s)
=
s-\frac{s^2}{2}.
\tag{7.5}\label{eq:subset-degree-bound}
\]
If $d_W^{+,\downarrow}$ denotes the decreasing rearrangement of $d_W^+$,
then, by the Hardy--Littlewood rearrangement principle,
\[
\int_0^s d_W^{+,\downarrow}(x)\,dx
=
\sup_{\lambda(S)=s}\int_Sd_W^+(x)\,dx.
\]
Therefore \eqref{eq:subset-degree-bound} gives
\[
\int_0^s d_W^{+,\downarrow}(x)\,dx
\le
s-\frac{s^2}{2}
=
\int_0^s(1-x)\,dx
\qquad(0\le s\le1).
\]
Applying Lemma~\ref{lem:continuous-majorization} with
$f=d_W^{+,\downarrow}$ and $g(x)=1-x$, and using
\eqref{eq:limit-square}, we obtain
\[
d_W^{+,\downarrow}(x)=1-x
\qquad\text{for almost every }x.
\tag{7.6}\label{eq:degree-rearrangement}
\]
In particular, $d_W^+$ is uniformly distributed on $[0,1]$ and has no atoms.

For $s\in[0,1]$, let
$S_s=\{x:d_W^+(x)>1-s\}$.
Since $d_W^+$ has the uniform distribution, we have $\lambda(S_s)=s$, and
\[
\int_{S_s}d_W^+(x)\,dx
=
s-\frac{s^2}{2}.
\]
Thus equality holds in both estimates used to prove
\eqref{eq:subset-degree-bound}. In particular,
$e^+(S_s,S_s^c)=s(1-s)$.
Since $0\le W^+\le1$ and $\lambda(S_s\times S_s^c)=s(1-s)$, it follows that
$W^+(x,y)=1$ for almost every $(x,y)\in S_s\times S_s^c$.

Taking these full-measure statements for all rational $s\in[0,1]$ and using
the fact that $d_W^+$ has no atoms, we obtain
\[
W^+(x,y)=1
\qquad\text{for almost every }(x,y)\text{ with }
d_W^+(x)>d_W^+(y).
\tag{7.7}\label{eq:degree-order-arcs}
\]
For almost every $x$, the set
$\{y:d_W^+(y)<d_W^+(x)\}$
has measure $d_W^+(x)$. By Fubini's theorem applied to
\eqref{eq:degree-order-arcs}, for almost every $x$ we have $W^+(x,y)=1$ for
almost every $y$ in this set; these arcs already account for the entire
out-degree of $x$. Hence
\[
W^+(x,y)=0
\qquad\text{for almost every }(x,y)\text{ with }
d_W^+(x)<d_W^+(y).
\tag{7.8}\label{eq:opposite-degree-order-no-arcs}
\]

Let $r(x)=1-d_W^+(x)$.
If $r(x)<r(y)$, then $d_W^+(x)>d_W^+(y)$, so
$W^+(x,y)=1$ and $W^+(y,x)=0$
for almost every such pair. Since
$W^+(x,y)=W_{10}(x,y)+W_{11}(x,y)$
and, by the symmetry relations,
$W^+(y,x)=W_{01}(x,y)+W_{11}(x,y)$,
we get $W_{10}(x,y)=1$ and
$W_{00}(x,y)=W_{01}(x,y)=W_{11}(x,y)=0$
for almost every pair with $r(x)<r(y)$. Similarly, the state is $01$ for
almost every pair with $r(x)>r(y)$. Thus the limiting colored graphon is
the pullback, through the rank map $r$, of the transitive-tournament
graphon; below we only use \eqref{eq:degree-order-arcs} and
\eqref{eq:opposite-degree-order-no-arcs}.

It remains to return from the limit object to finite edit distance. Fix
$K\in\mathbb{N}$ and partition the limit space into rank blocks
$U_p=\{x:\ (p-1)/K\le r(x)<p/K\}$, $1\le p\le K$. Since $r$ is uniformly
distributed, each $U_p$ has measure $1/K$. For $p<q$ and almost every
$(x,y)\in U_p\times U_q$ we have $r(x)<r(y)$, hence $W^+(x,y)=1$ and
$W^+(y,x)=0$ by \eqref{eq:degree-order-arcs} and \eqref{eq:opposite-degree-order-no-arcs}. Therefore
\[
e^+(U_p,U_q)=\frac{1}{K^2},\qquad e^+(U_q,U_p)=0\qquad(p<q).
\]
By Lemma~\ref{lem:rounding}, applied along the convergent subsequence, for
every fixed $K$ and all sufficiently large $i$ there is a partition
$V(D_i)=V_{i,1}\cup\dots\cup V_{i,K}$ with $|V_{i,p}|=\tfrac{m_i}{K}+o_K(m_i)$
such that, summed over all $p<q$, the number of missing forward arcs from
$V_{i,p}$ to $V_{i,q}$ together with the number of backward arcs from
$V_{i,q}$ to $V_{i,p}$ is $o_K(m_i^2)$. Here and below, $o_K(\cdot)$
denotes a quantity of smaller order as $i\to\infty$ for each fixed $K$.

List the vertices block by block, first all vertices of $V_{i,1}$, then all
vertices of $V_{i,2}$, and so on. Let $J_i$ be the ordered digon chain
obtained from this linear order by pairing consecutive vertices, with a
final singleton if $m_i$ is odd. Correcting all inter-block arcs costs
$o_K(m_i^2)$ arc changes. Correcting all arcs with both endpoints in the
same part $V_{i,p}$ costs at most
\[
2\sum_{p=1}^K\binom{|V_{i,p}|}{2}
\le
\frac{m_i^2}{K}+o_K(m_i^2).
\]
Finally, at most $K-1$ digon blocks of $J_i$ cross a boundary between
consecutive parts $V_{i,p}$ and $V_{i,p+1}$; the additional cost caused by
these crossing pairs is $O(K)$. Hence
\[
|A(D_i)\triangle A(J_i)|
\le
o_K(m_i^2)+\frac{m_i^2}{K}+o_K(m_i^2)+O(K).
\]
Choose $K$ sufficiently large and then $i$ sufficiently large. The right-hand
side is then smaller than $\rho_0m_i^2$, contradicting the choice of $D_i$.
This proves the lemma.
\end{proof}

\begin{lemma}\label{lem:chain-align}
For every $\rho>0$ there exists $\theta>0$ such that the following holds. If
$J$ and $J'$ are ordered digon chains on the same $m$-element set and
$\sum_y\bigl(d_J^+(y)-d_{J'}^+(y)\bigr)^2\le\theta m^3$,
then
\[
|A(J)\triangle A(J')|\le\rho m^2.
\]
\end{lemma}

\begin{proof}
For an ordered digon chain $K$, let $r_K(y)$ denote the index of the block
containing $y$, regarding the final singleton, when present, as the last
block. In both parity cases,
$d_K^+(y)=m-2r_K(y)+1$,
and hence
\[
|r_J(y)-r_{J'}(y)|
=
\frac12|d_J^+(y)-d_{J'}^+(y)|.
\tag{7.9}\label{eq:block-rank-distance}
\]

Choose a linear order associated with each chain by listing the blocks in
order and then listing the vertices inside each two-vertex block. We choose
the two possible internal orders block by block as follows: whenever a block
of $J$ and the block with the same index in $J'$ have a common vertex, we
place this common vertex in the same internal position in both orders. This
choice can be made independently for each block.

Let $p_J(y)$ and $p_{J'}(y)$ be the positions of $y$ in the two resulting
linear orders. If $r_J(y)=r_{J'}(y)$, then the choice of internal orders gives $p_J(y)=p_{J'}(y)$.
If $r_J(y)\ne r_{J'}(y)$, then the two positions differ by at most
\[
2|r_J(y)-r_{J'}(y)|+1
\le
3|r_J(y)-r_{J'}(y)|.
\]
Thus, for every $y$,
\[
|p_J(y)-p_{J'}(y)|
\le
3|r_J(y)-r_{J'}(y)|.
\tag{7.10}\label{eq:position-rank-bound}
\]

By an inequality of Diaconis and Graham \cite{DiaconisGraham}, the number of inversions between two linear orders (Kendall's tau distance) is at most the
sum of the position displacements (Spearman's footrule). Since one inversion reverses the orientation of one
unordered pair, it contributes two directed arcs to the symmetric difference
of the associated transitive tournaments. Hence these transitive tournaments
differ in at most
\[
2\sum_y |p_J(y)-p_{J'}(y)|
\le
6\sum_y|r_J(y)-r_{J'}(y)|
\]
arcs.

Each ordered digon chain is obtained from its associated transitive
tournament by adding one reverse arc inside each two-vertex block. Let
$M_J$ and $M_{J'}$ be these two added-arc sets. If an added arc belongs to
$M_J\triangle M_{J'}$, then it is incident with at least one vertex whose
block index is different in the two chains. Since each vertex is incident
with at most one added arc in each chain,
\[
|M_J\triangle M_{J'}|
\le
2|\{y:r_J(y)\ne r_{J'}(y)\}|
\le
2\sum_y|r_J(y)-r_{J'}(y)|.
\]
Consequently,
\[
|A(J)\triangle A(J')|
\le
8\sum_y|r_J(y)-r_{J'}(y)|
=
4\sum_y|d_J^+(y)-d_{J'}^+(y)|.
\]
Cauchy--Schwarz gives
\[
|A(J)\triangle A(J')|
\le
4\sqrt{m\sum_y(d_J^+(y)-d_{J'}^+(y))^2}
\le
4\sqrt\theta\,m^2.
\]
Taking $\theta\le(\rho/4)^2$ proves the lemma.
\end{proof}

\begin{theorem}\label{thm:edit-stability}
For every $\eta>0$ there is $\varepsilon>0$ such that the following holds
for every $m$. Let $\mathcal P=(\C,\A)$ be an $m$-color palette satisfying
$\mathcal P_L\npreceq\mathcal P$
and $\mathcal P_R\npreceq\mathcal P$, and suppose that
$|\A|\ge E_m-\varepsilon m^3$.
Then there is an ordered digon chain $J$ on $\C$ such that
\[
|\A(\mathcal P)\triangle\A(\mathcal P_m^\star(J))|\le\eta m^3.
\]
\end{theorem}

\begin{proof}
Fix $\eta>0$ and put $\delta=\eta/4$. Let $\theta_1>0$ be supplied by
Lemma~\ref{lem:chain-align} with $\rho=\delta$. Choose $\xi>0$ so small that
$2\xi+\delta\le\frac\eta2$ and
$6\xi\le\frac{\theta_1}{2}$.
Let $\theta_0>0$ be supplied by
Lemma~\ref{lem:C3-digraph-stability} with $\rho=\xi$. Finally, choose
$\varepsilon>0$ such that
$2\varepsilon\le\theta_0$,
$6\varepsilon\le\frac{\theta_1}{2}$, and $\varepsilon\le\frac\eta2$.

For each $y\in\C$, write
\[
B_y=\{(x,z)\in\C^2:(x,y,z)\in\A\},
\]
and let $a_y=d^-_{G_L(\mathcal P)}(y)$,
$b_y=d^+_{G_R(\mathcal P)}(y)$ and
$t_y=|B_y|$.
Since $E_m-|\A|\le\varepsilon m^3$,
Proposition~\ref{prop:defect} gives
\[
E_m-\sum_ya_y^2\le2\varepsilon m^3,
\qquad
E_m-\sum_yb_y^2\le2\varepsilon m^3,
\]
\[
\sum_y(a_y-b_y)^2\le2\varepsilon m^3,
\qquad
\sum_y(a_yb_y-t_y)\le\varepsilon m^3.
\tag{7.11}\label{eq:defect-components}
\]

Since $\mathcal P$ avoids $\mathcal P_L$ and $\mathcal P_R$, the auxiliary
digraphs $G_L(\mathcal P)$ and $G_R(\mathcal P)$ are loopless and
$\Cvec{3}$-free. Therefore $G_L(\mathcal P)^{\mathrm{rev}}$ and
$G_R(\mathcal P)$ satisfy the hypotheses of
Lemma~\ref{lem:C3-digraph-stability}. By the first two inequalities in
\eqref{eq:defect-components} and the choice of $\varepsilon$, this lemma
gives ordered digon chains $J_L,J_R$ on $\C$ such that
\[
|A(G_L(\mathcal P)^{\mathrm{rev}})\triangle A(J_L)|\le\xi m^2,
\qquad
|A(G_R(\mathcal P))\triangle A(J_R)|\le\xi m^2.
\tag{7.12}\label{eq:aux-close-to-chains}
\]

If two digraphs on the same $m$-vertex set differ in $r$ arcs, then the sum
of the squared differences of corresponding out-degrees is at most $mr$.
Using this observation, \eqref{eq:defect-components},
\eqref{eq:aux-close-to-chains}, and
\[
(x+y+z)^2\le3(x^2+y^2+z^2),
\]
we obtain
\[
\sum_y\bigl(d_{J_L}^+(y)-d_{J_R}^+(y)\bigr)^2
\le
6(\xi+\varepsilon)m^3
\le
\theta_1m^3.
\]
Lemma~\ref{lem:chain-align} therefore gives
$|A(J_L)\triangle A(J_R)|\le\delta m^2$.
Set $J=J_R$. Combining this with \eqref{eq:aux-close-to-chains}, we get
\[
|A(G_L(\mathcal P)^{\mathrm{rev}})\triangle A(J)|
+
|A(G_R(\mathcal P))\triangle A(J)|
\le
(2\xi+\delta)m^2.
\tag{7.13}\label{eq:both-aux-close}
\]

For each middle color $y$, let
$X_y=N^-_{G_L(\mathcal P)}(y)$,
$Y_y=N^+_{G_R(\mathcal P)}(y)$, and $Z_y=N_J^+(y)$.
Equation~\eqref{eq:both-aux-close} gives
\[
\sum_y|X_y\triangle Z_y|
+
\sum_y|Y_y\triangle Z_y|
\le
(2\xi+\delta)m^2.
\tag{7.14}\label{eq:neighborhood-close}
\]
For every $y$, we use the elementary bound
\[
|(X_y\times Y_y)\triangle(Z_y\times Z_y)|
\le
m|X_y\triangle Z_y|+m|Y_y\triangle Z_y|.
\]
Summing over $y$ and applying \eqref{eq:neighborhood-close}, we obtain
\[
\sum_y|(X_y\times Y_y)\triangle(Z_y\times Z_y)|
\le
(2\xi+\delta)m^3.
\tag{7.15}\label{eq:product-close}
\]

The middle-$y$ slice $B_y$ of $\mathcal P$ is contained in
$X_y\times Y_y$, and \eqref{eq:defect-components} gives
\[
\sum_y|(X_y\times Y_y)\setminus B_y|
=
\sum_y(a_yb_y-t_y)
\le
\varepsilon m^3.
\tag{7.16}\label{eq:slice-missing}
\]
On the other hand, the middle-$y$ slice of
$\mathcal P_m^\star(J)$ is exactly $Z_y\times Z_y$. Therefore
\[
|\A(\mathcal P)\triangle\A(\mathcal P_m^\star(J))|
\le
\sum_y|(X_y\times Y_y)\triangle(Z_y\times Z_y)|
+
\sum_y|(X_y\times Y_y)\setminus B_y|.
\]
Using \eqref{eq:product-close} and \eqref{eq:slice-missing}, we get
\[
|\A(\mathcal P)\triangle\A(\mathcal P_m^\star(J))|
\le
(2\xi+\delta+\varepsilon)m^3
\le
\eta m^3,
\]
as required.
\end{proof}
\section{Palette Lagrangians}\label{sec:lagrangian}
Given a palette $\mathcal P=(\C,\A)$, define its \emph{Lagrangian polynomial} by
\[
\lambda_{\mathcal P}(\mathbf{x}) = \sum_{(i,j,k) \in \A} x_i x_j x_k,
\qquad \mathbf{x}=(x_i)_{i\in\C}.
\]
The \emph{Lagrangian} of $\mathcal P$, denoted by $\Lambda(\mathcal P)$, is the maximum of $\lambda_{\mathcal P}(\mathbf{x})$ subject to $\sum_{i\in\C}x_i=1$ and $x_i\in[0,1]$ for all $i\in\C$. Let $\Lambda_{\text{pal}}$ be the set of all possible values of $\Lambda(\mathcal P)$.

\begin{theorem}[{King--Piga--Sales--Sch\"ulke~\cite[Theorem~1.1]{KingPigaSalesSchulke2025}}]\label{thm:KPSS}
For every $\lambda \in \Lambda_{\text{pal}}$, there is a finite family $\mathcal{F}$ of $3$-graphs with $\piu(\mathcal{F})=\lambda$.
\end{theorem}

For the endpoint palette $\mathcal{P}_m^\star=\mathcal{P}_m^\star(\Fchain_{m,2})$, we have the following theorem.

\begin{theorem}\label{thm:lagrangian}
For every $m\ge2$, we have $\Lambda(\mathcal{P}_m^\star)=\frac13-\frac1{3m^2}$.
The unique maximizer is the uniform probability vector on the $m$ colors.
\end{theorem}

\begin{proof}
Let $\mu$ be an arbitrary feasible probability vector on the color set
of $\mathcal{P}_m^\star$.
Assume first that $m=2s$. Let the two weights in the $i$-th digon block be $u_i,v_i$, set $b_i=u_i+v_i$, and let
$T_i=\sum_{j>i}b_j$.
The contribution of this block is
\[
u_i(T_i+v_i)^2+v_i(T_i+u_i)^2
=b_iT_i^2+(4T_i+b_i)u_iv_i\le b_iT_i^2+(4T_i+b_i)\frac{(u_i+v_i)^2}{4}.
\]
For fixed $b_i$, this contribution is maximized when $u_i=v_i=b_i/2$,
and the maximum is
$b_i\left(T_i+\frac{b_i}{2}\right)^2$.
Writing $s_i=\sum_{j\ge i}b_j$, we have
\[
b_i\left(T_i+\frac{b_i}{2}\right)^2
=\frac13(s_i^3-s_{i+1}^3)-\frac{b_i^3}{12}.
\]
Summing gives
\[
\lambda_{\mathcal{P}_m^\star}(\mu)\le \frac13-\frac1{12}\sum_{i=1}^s b_i^3.
\]
Since $\sum_i b_i=1$, Jensen gives $\sum_i b_i^3\ge 1/s^2$, and hence
\[
\lambda_{\mathcal{P}_m^\star}(\mu)\le \frac13-\frac1{12s^2}=\frac13-\frac1{3m^2}.
\]
Equality requires $b_i=1/s$ for all $i$ and $u_i=v_i=b_i/2$, i.e.\ uniform weights.

If $m=2s+1$, let $c$ be the weight of the final singleton, let
$b_1,\ldots,b_s$ be the digon-block weights, and define
$T_i=c+\sum_{j>i}b_j$.
The same block calculation, now with the singleton included in every tail,
gives
\[
\lambda_{\mathcal{P}_m^\star}(\mu)
\le \frac{1-c^3}{3}-\frac1{12}\sum_{i=1}^s b_i^3
\le \frac{1-c^3}{3}-\frac{(1-c)^3}{12s^2}.
\]
The right-hand side has derivative
$-c^2+\frac{(1-c)^2}{4s^2}$,
so its maximum on $[0,1]$ occurs uniquely at $c=\frac{1}{2s+1}$. Then
\[
\frac{1-c^3}{3}-\frac{(1-c)^3}{12s^2}=\frac{4s(s+1)}{3(2s+1)^2}=\frac13-\frac1{3m^2}.
\]
Equality also requires equality in Jensen's inequality and in each block estimate, so $b_i=(1-c)/s$ and $u_i=v_i=b_i/2$ for all $i$. Thus the unique maximizer is again the uniform vector.
\end{proof}

\begin{proposition}\label{prop:lag-defect}
For $m=2s$, with notation as in the proof of Theorem~\ref{thm:lagrangian},
\[
\frac13-\frac1{3m^2}-\lambda_{\mathcal{P}_m^\star}(\mu)
=
\frac1{12}\left(\sum_{i=1}^s b_i^3-\frac1{s^2}\right)
+\frac14\sum_{i=1}^s(4T_i+b_i)(u_i-v_i)^2.
\]
For $m=2s+1$ the analogous identity is
\[
\frac13-\frac1{3m^2}-\lambda_{\mathcal{P}_m^\star}(\mu)
=
\left(\frac{c^3}{3}+\frac1{12}\sum_{i=1}^s b_i^3-\frac1{3m^2}\right)
+\frac14\sum_{i=1}^s(4T_i+b_i)(u_i-v_i)^2,
\]
where $T_i=c+\sum_{j>i}b_j$ in the odd case. Both right-hand sides are
nonnegative and vanish only at the uniform vector.
\end{proposition}

\begin{proof}
The identities are obtained by retaining the exact error term
\[
u_i(T_i+v_i)^2+v_i(T_i+u_i)^2
=b_i\left(T_i+\frac{b_i}{2}\right)^2
-\frac{4T_i+b_i}{4}(u_i-v_i)^2.
\]
For even $m$, the remaining scalar term is nonnegative by Jensen's
inequality. For odd $m$, Jensen's inequality reduces it to
\[
\frac{c^3}{3}+\frac{(1-c)^3}{12s^2}-\frac1{3(2s+1)^2},
\]
which is nonnegative and vanishes uniquely at $c=1/(2s+1)$. The equality
conditions in the block estimates and in Jensen's inequality then give the
uniform vector in both cases.
\end{proof}

\begin{corollary}
$\frac{1}{3}$ is an accumulation point of uniform Tur\'an densities of
finite forbidden families.
\end{corollary}

\begin{proof}
This follows from Theorems~\ref{thm:KPSS} and~\ref{thm:lagrangian} by letting $m\to\infty$.
\end{proof}

\section{Open problems}\label{sec:questions}

Theorem~\ref{thm:intro-palette} reduces a family of finite palette extremal problems to degree-square Tur\'an problems for digraphs. We evaluated this degree-square problem for directed cycles and for the transitive triangle, but the next cases already appear to be substantially more delicate.

\begin{problem}
For which self-converse tournaments $T$ can one determine $\operatorname{ex}_2^+(m,T)$ exactly?
\end{problem}

A particularly natural test case is the strongly connected tournament on four vertices. Let $R_4$ be the tournament with score sequence $(1,1,2,2)$. Unlike the cyclic-triangle case, the ordered complete-block chain need not be extremal for this problem: for $m=5$, the complete tripartite digraph with parts of sizes $\{1,2,2\}$ is $R_4$-free and has degree-square sum $52$, whereas $\Fchain_{5,3}$ has degree-square sum $50$.

\begin{problem}
Determine $\operatorname{ex}_2^+(m,R_4)$, where $R_4$ is the strongly connected self-converse tournament on four vertices.
\end{problem}


\end{document}